\documentclass[11pt]{amsart}

\usepackage{graphicx}
\usepackage{amsmath}
\usepackage{amsfonts}
\usepackage{amscd}

 \usepackage{pifont}

\newtheorem{thm*}{\bf Theorem}[section]

\newtheorem{thm}{\bf Theorem}[section]
\newtheorem{cor}{\bf Corollary}[section]
\newtheorem{lemma}{\bf Lemma}[section]
\newtheorem{prop}{\bf Proposition}[section]
\newtheorem{defn}{\bf Definition}[section]

\newtheorem{rem}{\bf Remark}[section]
\newtheorem{nota}{\bf Notation}[section]

\newcommand{\<}{\langle}
\renewcommand{\>}{\rangle}
\newcommand{\D}{\text{d}}

\newcommand{\tr}{\operatorname{tr}}

\newcommand{\Res}{\operatorname{Res}}

\renewcommand{\phi}{\varphi}

\newcommand{\A}{\operatorname{\mathfrak{a}}}

\newcommand{\B}{\operatorname{\mathfrak{b}}}
 
\renewcommand{\Im}{\operatorname{Im}}
\renewcommand{\Re}{\operatorname{Re}}

\newcommand{\RR}{{\mathbb R}}
\newcommand{\CC}{{\mathbb C}}

\newcommand{\NN}{{\mathbb N}}

\newcommand{\ZZ}{{\mathbb Z}}

\begin{document}

 \title{
\Large{ Selberg  Zeta function and hyperbolic Eisenstein series.}
}

\author{ Th\'er\`ese \textsc{Falliero}}
\address{ Universit\'e d'Avignon et des Pays de Vaucluse,
Laboratoire de math\'ematiques d'Avignon (EA 2151),
F-84018 Avignon, France}

\email{therese.falliero@univ-avignon.fr}



\maketitle

%

\section{Introduction}


The relation between the geometric and the spectral properties of a Riemann surface (e.g., volume, periodic geodesics, etc., among the geometric objects, and eigenvalues, resonances, automorphic functions, etc., among the spectral entities) is an important subject with a long, rich history.

Here we consider a family of degenerating geometrically finite hyperbolic surfaces $(S_l)$ with one or more non separating geodesics or geodesics of a funnel being pinched.

The aim of the paper is to study the behavior of the hyperbolic Eisenstein series of $(S_l)$ through degeneration (for results on
degenerating Eisenstein series, see, for example \cite{o2},
\cite{o4}, \cite{jorg}, \cite{pippich}), on the left of  the critical axis $\{ s\in \CC, \Re s=1/2\}$.
For purposes of illustration, it is desirable to restrict ourselves to the weight 0 case.
We plan to devote one later paper to a treatment of more general weights.

Because of the behavior of the eigenvalues through degeneration (see \cite{jit}), it is not possible to define a good notion of convergence for hyperbolic Eisenstein series in any domain that intersects the critical axis  (at least if the surfaces $S_l$ are compact). To overthrow this problem we consider the product of  the Selberg Zeta function and  hyperbolic Eisenstein series which gives a holomorphic function on $\Re s=1/2$ (except eventually at $s=1/2$).

More precisely, let $S_l=\Gamma_l\backslash H$ be a family of geometrically finite hyperbolic surfaces 
degenerating to the surface $S$ with only one geodesic $c_l$
being pinched.  The group  
  $\Gamma_l$ contains the transformation
$\sigma_l(z)=e^lz$ corresponding to $c_l$ and the left half-collar for $c_l$ is in  $S_l\backslash F_1$. Let $S_0=\Gamma\backslash H$ be the component of $S$
\begin{enumerate}
\item  containing, in the non separating case,
$p$,  the cusp arising from the right  half-collar for $c_l$; $q$  will denote the cusp arising from the left  half-collar.
\item   containing, in the funnel's geodesic boundary case, the compact core, the cusp will be denoted by $\infty$.
\end{enumerate}
 Let $Z_l(s)$ be the Selberg Zeta function of $S_l$, $z_l(s) $   
	these of the hyperbolic cylinder $\<\sigma_l\>\backslash H$ 
and $E_l(z,s)$ the hyperbolic Eisenstein series associated to $c_l$.

Let denote by 
$$E^*_l(z,s)=\frac{Z_l(s)}{z_l(s)}\frac{1}{l^s}E_l(z,s)\, .$$
Consider two cases
\begin{enumerate}
\item $c_l$ is not separating,
\item $c_l$ is the geodesic boundary of  a funnel.

\end{enumerate}
We know by (\cite{schulze}, \cite{preprint} and \cite{hal}) that the family $(E^*_l(.,s))_l$ converges uniformly on every compact  subsets of $\Re s>1/2$ and $S_0$ to,
\begin{enumerate}
\item $Z_0(s)(E_p(.,s)+E_q(.,s))$ in the separating case,
\item $Z_0(s) E_{\infty}(.,s)$ in the funnel geodesic boundary case,
\end{enumerate}
 where   $Z_0(s)$ corresponds to the Selberg Zeta function of $S_0$.

\medskip
We summarize the main results:
\begin{enumerate}
\item In the finite volume case:
\begin{itemize}
\item  the meromorphic family of weighted hyperbolic Eisenstein series $(\frac{1}{l^s}E_l(.,s))_l$ converges uniformly on every compact  subsets of $S_0$ and on $0<\Re s<1/2$, to 0; 
\item the family $(E^*_l(.,s))_l$ doesn't converge on $\Re s>0$;
\item in the non compact case, the scattering matrix and  the (classical)  Eisenstein series converge for $\Re s<1/2$;
\item the weighted Selberg Zeta function  $\displaystyle\frac{Z_l(s)}{z_l(s)}$, doesn't converge on $0<\Re s <1/2$ even if  the $S_l$ are non compact.
\end{itemize}
\item In the infinite volume case, with $c_l$ the geodesic boundary of a funnel:
\begin{itemize}
\item  for $\Re s >1/2$, $2(1-s)\not = 0, -1,...$,in particular  for $1/2<\Re s <3/2$, the relative scattering determinant, $\tau_l(s)$, tends to $\displaystyle\frac{1}{2\sin^2\frac{\pi}{2}s}$ as $l$ tends to 0;
\item  the weighted Selberg Zeta function doesn't converge on $0<\Re s <1/2$;
\item  the meromorphic family of weighted hyperbolic Eisenstein series $(\frac{1}{l^s}E_l(.,s))_l$ converges  uniformly on every compact  subsets of $S_0$ and on $0<\Re s<1/2$, to 0;
\item  the family $(E^*_l(.,s))_l$ converges uniformly on every compact  subsets of $S_0$ and  $\Re s>0$, to  $Z_0(s) E_{\infty}(.,s)$.
\end{itemize}

\end{enumerate}

\begin{rem}
\begin{enumerate}
\item
In particular we answer a remark of Schulze (\cite{schulze}, p.122).
\item We can also note the parallel with his result (\cite{schulze}, Theorem 40).
\item The domains in the last two points of the preceding main results, are not the best, but in this introduction,  we let them because our aim is just to show,  we can pass the critical axis.
\end{enumerate}
\end{rem}

An application expected would be, in particular, to obtain some characterization of embedded eigenvalues. More precisely, an idea was to exploit the fact that (classical)  Eisenstein series don't have any pole on $\Re s =1/2$, while hyperbolic Eisenstein series have simple poles on $\Re s =1/2$.

In view of the summarized results, such application doesn't exist in the finite volume case. In the infinite volume case, to achieve such a characterization we need, in particular, to control the multiplicity of a resonance. If we had such a result as Corollary 4.2 in \cite{w2}, then an answer could be given. 
\section{Notations}\label{not}
We recall here, briefly, some notations.

A family of degenerating geometrically finite hyperbolic surfaces  consists of a surface $M$ and a smooth family $(g_l)_{l>0}$ of Riemannian metrics that meet the following assumptions:
\begin{enumerate}
\item The Riemannian manifold $M_l=(M,g_l)$ is a geometrically finite hyperbolic surface  for each $l$.
\item  There are
 finitely many disjoint open subsets  that are diffeomorphic to cylinders $ {\mathbb R}\backslash {\mathbb Z}\times J_i$ where $J_i\subset \mathbb R$ is a connected neighborhood of $0$, of cuspidal and funnel types. A cuspidal type  contains a cusp $C_i$, $1\leq i \leq n_c$, which is isometric to the half parabolic cylinder $C_\infty=([0,\infty[)_r \times ({\mathbb R}/ {\mathbb Z})_x$ with the metric $ds^2=dr^2+e^{-2r}dx^2$. A funnel type contains a funnel $F_j$, $1\leq j\leq n_f$, which is isometric to a  half  hyperbolic cylinder $\Gamma_{l_j}\backslash H$, $F_{l_j}=({\mathbb R}^+)_r\times ({\mathbb R}/ {\mathbb Z})_x$
with the metric $ds^2=dr^2+{l_j}^2\cosh^2(r)dx^2$ and $l_j$ is the lenght of the geodesic boundary.
\item There are
 finitely many disjoint open subsets ${\mathcal C}_{l,i}\subset M$ that are diffeomorphic to cylinders $ {\mathbb R}\backslash {\mathbb Z}\times J_{l,i}$ where $J_{l,i}\subset \mathbb R$ is a connected neighborhood of $0$ with the metric  $(x,a)\mapsto (l_i(l)^2+ a^2)dx^2+((l_i(l)^2+ a^{2})^{-1}da^2$ and $l_i(l)\to 0$ as $l\to 0$. The curve $c_i={\mathbb R}\backslash {\mathbb Z}\times \{0\}$ is a closed geodesic of length $l_i(l)$.
\item The complement of $(C_1\cup...\cup C_{n_c})\cup (F_1\cup ...\cup F_{n_f})\cup_i {\mathcal C}_{l,i}$, where we may have some $ F_j\subset {\mathcal C}_{l,i}$ is relatively compact.
\item On $M_0:=M\backslash \cup_i c_i$, the metrics $g_l$ converge smoothly to a hyperbolic metric $g_0$ as $l\to 0$. $M_0$ is a possibly non connected hyperbolic surface that contains a pair of  cusps for each $i$.
\end{enumerate}
 We set
$ M=K\cup C\cup F\, ,$
where $C=C_1\cup...\cup C_{n_c}$ and $F=F_1\cup ...\cup F_{n_f}$.

The standard  funnel of diameter $l>0$, $F_l$, is the half  hyperbolic cylinder $\<z\mapsto e^l z\>\backslash H$, $F_l=({\mathbb R}^+)_r\times ({\mathbb R}/ {\mathbb Z})_x$
with the metric $ds^2=dr^2+l^2\cosh^2(r)dx^2$, with $\<z\mapsto e^l z\>$ the cyclic
group generated by the transformation $z\mapsto e^l z$.

 The standard cusp $C_\infty$ is the half parabolic cylinder $\Gamma_\infty\backslash H$, $C_\infty=([0, \infty[)_r\times ({\mathbb R}/ {\mathbb Z})_x$ with the metric $ds^2=dr^2+e^{-2r}dx^2$.

A standard collar for a geodesic of length $l$ is a cylinder
isometric to $ \<z\mapsto e^l z\>\backslash{\mathcal C}$ with
${\mathcal C}=\{ z=re^{i\theta}, 1\leq r\leq e^l,
l<\theta<\pi-l\}\subset H$ with the restriction of the
hyperbolic metric. There is a constant $k_0$ (the
short geodesic constant) such that each closed geodesic on $M_l$
of length at most $k_0$ has a neighbourhood isometric to the
standard collar and each cusp for $M_l$ has a neighbourhood
isometric to the standard cusp; furthermore, the collars for
short geodesics and the cusp regions are all mutually disjoint.

We define the function $r$ as the distance to the
compact core $K$ and the function $\rho$ by
\begin{equation}\label{standard}
\rho(r)=\left\{\begin{array}{cc}2e^{-r}&\text{ in $ F$}\\ e^{-r}&\text{ in $C$}\end{array}\right. , 
\end{equation}
with $\rho$ extended to a smooth non vanishing function inside $K$ in some arbitrary way.
We will adopt $(\rho,t)\in(0,2]\times {\RR}/l_j{\ZZ}$ as the standard
coordinates for the funnel $F_j$, where $t$ is the arc length around
the central geodesic at $\rho=2$.\\ For the cusp, our standard
coordinates $(\rho, t)\in (0,1]\times {\RR}/{\ZZ}$ are based on the
model defined by the cyclic group $\Gamma_\infty$. The cusp boundary is $y=1$, so that $y=e^r$ and
$\rho=1/y$. We set  $t \equiv x \pmod \ZZ$.

Here we consider a family of surfaces $S_l=\Gamma_l\backslash H$ degenerating to the surface $S$ with only one geodesic $c_l$ being pinched to form a pair of cusps on $S$.
As is common, we realize $H$ as $\{z\in\CC, \Im z>0\}$. Writing $z=x+iy$, then the hyperbolic metric $ds^2$ and hyperbolic Laplacian acting on function can be expresses as
\begin{equation}\label{lap}
 ds^2=\frac{dx^2+dy^2}{y^2}\quad \text{  and } \quad \Delta_l=y^2\left( \frac{\partial^2}{\partial x^2} +\frac{\partial^2}{\partial y^2}\right),
\end{equation}
$\Gamma_l$ containing the transformation $\sigma_l(z)=e^l z$ corresponding to $c_l$. 
 We denote by $p$ and $q$ the two cusps of $S$ arising from pinching $c_l$:  $p$ the limit of the right side of the $c_l$-collar and $q$ the limit of the left side of the $c_l$-collar.
Many mathematicians investigated, by various parameterizations, degeneration of hyperbolic surfaces and the asymptotic behavior of several functions. Basically those various parameterizations turn out to be almost the same powerful tools
(see Remark \ref{parame}).
We will start with \cite{w2}.  Let $K_l$ be $S_l$
minus $C_l$ the standard collar for $c_l$. There exist
homeomorphisms $f_l$ from $S_l\backslash c_l$ to $S$, with
$f_l$ tending to isometries $C^2$-uniformly on the compact core
$K_l \subset S_l$; define $\pi_l=f_l^{-1}$.
  Let
$S_0=\Gamma\backslash H$ be the component of $S$ containing
$p$ and conjugate $\Gamma$ to represent the cusp by the
translation $w\mapsto w+1$. In the following, $p=\infty$. In a similar way, without loss of generality we can represent  $q$ by  $0$.\\
Here, we will opt, as definition of the hyperbolic Eisenstein series, that which one can find in \cite{kramer}.

The exponent of convergence
$\delta$ of a Fuchsian group $\Gamma$ is defined to be the
abscissa of convergence of the Dirichlet series
$$\delta=\inf\{  s> 0, \sum_{T\in \Gamma} e^{-s d(z,Tw)} < \infty\}$$
for some $z,w\in H$, where $d(z,w)$  denotes the hyperbolic distance from $z\in H$ to $w\in H$.

 Let $\delta_l$ be the exponent of convergence for  $\Gamma_l$, then for $\Re s>\delta_l$:

\begin{equation}\label{formule}
E_l(z,s)=\sum_{\<\sigma_l\>\backslash \Gamma_l}\sin ^s(\theta \gamma z)\, .
\end{equation}
A very quickly way to verify the convergence for $\Re s>\delta_l$, of the right hand series is to use Proposition 13 of \cite{wave}, which remains true in the infinite volume case.

Recall  the definition of the (classical) Eisenstein series. We will assimilate a parabolic point with its cusp.
The stabilizer of a cusp $\A$ is an infinite cyclic group generated by a parabolic motion,
$$\Gamma_{\A}=\{ \gamma \in \Gamma : \gamma \A=\A\}=\< \gamma_{\A} \>\ ,$$
say. There exists a $\sigma_{\A} \in SL_2(\mathbb{R})$, called  a scaling matrix of the cusp $\A$,  such that
$\sigma_{\A}\infty=\A,\ \sigma_{\A}^{-1}\gamma_{\A} \sigma_{\A}=\left( \begin{array}{cc}1&1\\ 0 &1\end{array}\right) $;
$\sigma_{\A}$ is determined up to composition with a translation from the right. 
For $Y\geq 1$, the semi-strip $C_\infty(Y)=\{z=x+iy, 0<x<1, y\geq Y\}$ is mapped into the cuspidal zone $C_p(Y)=\sigma_p C_\infty(Y)$ with length $1/Y$ horocycle.

The Eisenstein series for the cusp $\A$ is then defined by
$$E_{\A}(z,s)=\sum_{\Gamma_{\A}\backslash \Gamma}y(\sigma_{\A}^{-1}\gamma z)^s\ ,$$
where $s$ is a complex variable with $\Re s >\delta_l$.
$E_{\A}(z,s)$ has Fourier expansion at the cusp $\B$ (see for example, \cite{iwaniec} Theorem 3.4, \cite{h2} Proposition 8.6, \cite{k} Section 2.2, \cite{ven} Section 3.1, \cite{fayreine})

\begin{equation}
E_{\A}(\sigma_{\B}z,s)=\delta_{\A\B} y^s +\phi_{\A\B}y^{1-s}+\sum_{n\not= 0}\phi_{\A\B}(n,s)W_s(nz)\,  ,
\end{equation}
$W_s(z)$ is expressed in terms of the Whittaker function, $W_{0,s-\frac{1}{2}}$ by: $W_s(z)= e^{2i\pi n x}W_{0,s-\frac{1}{2}}(4\pi |n| y)$, $z=x+iy$.

Whittaker functions $W_{0,s-\frac{1}{2}}$, $M_{0,s-\frac{1}{2}}$ are defined in terms of the modified Bessel functions $I$ and $K$, in the following way (see \cite{abra}, Sections 9.6-9.7; \cite{fayreine}, p.172):
$$W_{0,s-\frac{1}{2}}(t)=\left(\frac{t}{\pi}\right)^{\frac{1}{2}}K_{s-\frac{1}{2}}\left(\frac{t}{2}\right),\qquad
M_{0,s-\frac{1}{2}}(t)=\Gamma\left(s+\frac{1}{2}\right)4^{s-\frac{1}{2}}t^{\frac{1}{2}}I_{s-\frac{1}{2}}\left(\frac{t}{2}\right)\, .$$

 $K_{s-1/2}(t)$ the MacDonald-Bessel function,  has the following asymptotics, independent of the parameter $s$,
$$ K_{s-\frac{1}{2}}(t)\sim\sqrt{\frac{\pi}{2 t}} e^{-t}, \text{ as } t\to +\infty, \quad 
 I_{s-\frac{1}{2}}(t)\sim  \frac{e^t}{\sqrt{2\pi t}} , \text{ as } t\to +\infty\, .$$
Moreover we have then the asymptotic behaviors: 
\begin{eqnarray}\label{whit1}
W_{0,s-\frac{1}{2}}(t)&\sim & e^{-t/2} \qquad t\to\infty\, ;\\ \label{whit2}
M_{0,s-\frac{1}{2}}(t)&\sim &\Gamma(s+\frac{1}{2})\dfrac{4^{s-\frac{1}{2}}}{\sqrt{\pi}} e^{t/2} \qquad t\to\infty\, .
\end{eqnarray}

The Eisenstein series initially defined as a serie for $\Re s>1$ can be continued meromorphically to the all $s$-plane. The poles of $E_{\A}(z,s)$ in $\Re s >1/2$ are among the poles of $\phi_{\A\A}(s)$ and they are simple.

The Selberg Zeta function is defined for $\Re s>\delta_l$ by an absolutely convergent product 

$$Z_l(s)=\Pi_c z(l_c,s), \qquad \text{where   }  z(l_c,s)=\Pi_{k=0}^\infty(1-e^{-(s+k)l_c})^2\,.$$

The first product ranges over the set of all unoriented, primitive, closed geodesics $c$ of the surface, where $l_c$ denotes its length. $Z_l(s)$ extends meromorphically to $s\in \CC$ (see for example, \cite{borth}, Chapter 10). Note that $z(l_c,s)$ is an entire function with zeros $\{n\in\ZZ, n\leq 0\}$. When $c=c_l$, the pinched geodesic, we will denote by $z_l(s)=z(l_{c_l},s)$.

We review basic material on the expansion of eigenfunctions. Our essentials references are \cite{fayreine}, \cite{h2}.

A fundamental domain in $H$ for the cyclic sungroup $\Gamma_l$ generated by $\sigma_l$ is the semi annulus:
\begin{equation}\label{coor}
z=r e^{i\theta} \qquad 1\leq r> e^l, \, 0<\theta <\pi
\end{equation}
which is homeomorphic to the annulus $e^{-2\pi^2/l}<|t|<1$ under the conformal projection
$$t=z^\frac{2\pi i}{l}: H\longrightarrow {\CC}-\{0\}\, .$$
We consider the Fourier expansion of an eigenfunction $f$, $\Delta_l f+s(1-s)f=0$, for the hyperbolic Laplacian $\Delta_l$, $f$ invariant for $\sigma_l$.
$$f(z)=\sum_n f_n(\theta)r^{\bar{n}}, \quad \bar{n}=2i\pi n/l$$
in the coordinates (\ref{coor}), then
\begin{equation}\label{eqdif}
\frac{\D^2f_n}{\D\theta^2}+\left\{ \frac{s(1-s)}{\sin^2\theta}+\bar{n}^2\right\}f_n=0\, .
\end{equation}
Setting $f_n(\theta)=g_n(u), u=i\cot\theta$, then
$$\frac{\D^2g_n}{\D u^2}+ \frac{2u}{u^2-1}\frac{\D g_n}{\D u}+\left\{ \frac{s(1-s)}{u^2-1}-\bar{n}^2\frac{1}{(u^2-1)^2}\right\}g_n=0\, .$$
Fay observes that solutions can be given in terms of the  Gauss hypergeometric function $F$. We follow him (\cite{fayreine}, Formula (87)) with the
\begin{defn}\label{2.1}
We let
\begin{equation}\label{def1}
N_s^n(\theta)=(\sin\theta)^s\exp\left[i\theta\left(s+\bar{n}\right)\right]F\left(s+\bar{n},s,2s;-2i\sin\theta e^{i\theta}\right)
\end{equation}
be the unique solution to $ (\ref{eqdif})$  for which
\begin{equation}\label{def2}
\lim_{\theta\to 0}\theta^{-s}N_s^n(\theta)=\lim_{\theta\to \pi}(\pi -\theta)^{-s}N_s^n(\pi -\theta)=1\, .
\end{equation}
$N_s^n(\theta)$ is analytic for all $2s\not =-1,...$.
\end{defn}
To keep in mind:
\begin{nota}\label{not}
We will denote $R_l(s)$ the resolvent operator defined for $\Re s>1/2$, $s\not\in [1/2,1]$ by 
$R_l(s)=(\Delta_l+s(1-s))^{-1}$. 

Its kernel will be denoted by $G_s^l(z,w)$.
\end{nota}
In the following, we will need the
\begin{defn}\label{bound}
A family $(f_n)$ of meromorphic functions on a domain $\mathcal O$ in $\mathbb C$ is called bounded in $\mathcal O$ if
\begin{enumerate}
\item $\exists (z_m)_m$ a discrete subset in $\mathcal O$,
\item $\forall K$ compact set of ${\mathcal O}\backslash \{z_m\}$, $\exists n(K)$, $\forall n\geq n(K)$, $f_n$ has no pole in $K$,
\item $M_K=\sup_{n\geq n(K)}(\sup_{z\in K}|f_n(z)|)< +\infty$.
\end{enumerate}
\end{defn}

\begin{defn}\label{bound}
A family $(f_n)$ of meromorphic functions on a domain $\mathcal O$ in $\mathbb C$ is called bounded in $\mathcal O$ if
\begin{enumerate}
\item $\exists (z_m)_m$ a discrete subset in $\mathcal O$,
\item $\forall K$ compact set of ${\mathcal O}\backslash \{z_m\}$, $\exists n(K)$, $\forall n\geq n(K)$, $f_n$ has no pole in $K$, and the family $(f_n)_{n\geq n(K)}$ converges uniformly on $K$.
\end{enumerate}
\end{defn}

\medskip

For degenerating surfaces the convergence of small eigenvalues, that is eigenvalues in the range $[0,1/4)$ is well understood (\cite{coco}, \cite{he}, \cite{jit}, \cite{w2}, \cite{schulze}). This result, throughout this paper, is implicitly used.

\section{The compact case}
To be more precise, we return first  to the case of a family of degenerating {\it  compact} Riemann surfaces $(S_l)$ and we show, with the notations of Section \ref{not} 
\begin{prop}\label{3.1}
 The weighted hyperbolic Eisenstein series $(\frac{1}{l^s}E_l(z,s))_l$ converges uniformly on every compact  subsets of $S_0$ and $0<\Re s<1/2$, to zero, as a meromorphic family.
\end{prop}
Before proving this proposition, we need to recall some results on hyperbolic Eisenstein series. In the compact case $\delta_l=1$ and the hyperbolic Eisenstein series  is first given as a convergent series for $\Re s>1$ (Formula (\ref{formule})). The analytic continuation of $E_l$ has been carried on in \cite{kmi}, \cite{ris}, \cite{kramer}. Succently speaking it follows from the relation that links the hyperbolic Eisensein series to the resolvent:
$$\Delta_l E_l(z,s)+s(1-s)E_l(z,s)=s^2E_l(z,s+2)\, .$$
In the compact case
there exists a complete orthonormal basis $\{\psi_j\}$ for $L^2(S_l)$ such that  $\Delta_l \psi_j =\lambda_j \psi_j$, with 
$0=\lambda_0<\lambda_1\leq \lambda_2\leq ...\rightarrow \infty$.

The resolvent  $R_l(s):L^2(S_l)\rightarrow L^2(S_l)$ has a meromorphic continuation to the whole complex plane, more precisely it's analytic for $s(1-s)\not =\lambda_j$ where  $\lambda_j$ is an eigenvalue of the Laplacian, for $\lambda\not = \lambda_j$, $G_s^l(z,w)=\sum_{n=0}^\infty \frac{\psi_n(z)\overline{\psi_n(w)}}{\lambda-\lambda_n}$. The hyperbolic Eisenstein series admits  the spectral expansion
$$E_l(z,s)=\sum_{j=0}^\infty a_{j,l}(s) \psi_j(z)\, .$$
The coefficient $a_{j,l}$ is given by the formula
$$a_{j,l}(s)=\sqrt{\pi}\frac{\Gamma((s-\frac{1}{2}+ir_j)/2)\Gamma((s-\frac{1}{2}-ir_j)/2)}{\Gamma(s/2)^2}\int_{c_l} \psi_j(z)\, ds(z)\, ,$$
here we have written the eigenvalue $\lambda_j$ of the eigenfunction $\psi_j$ in the form $\lambda_j=\frac{1}{4}+r_j^2$, $(\psi_j)_{j\geq 0}$ is a complete orthonormal system of smooth eigenfunctions. In conclusion (\cite{ris}), the function $E_l(z,s)$ has meromorphic continuation to the whole $s$ plane. At a regular point $s_0$, $E_l(z,s)$ is square integrable on $\Gamma_l\backslash H$. The eventual poles are located at $-2n+s_j$ where $s_j(1-s_j)$ is an eigenvalue of the automorphic Laplacian and $n\in \NN$. They are simple, except in the eventual case $s_j=1/2$, of order 2. The pole at $s=1$ is simple with residue
$$\frac{2l}{\text{vol}(\Gamma_l\backslash H)}\,.$$

What can leave us pressing Proposition \ref{3.1}, are the following facts:
\begin{lemma}
\begin{enumerate}
\item The family $(\frac{1}{l^s}E_l(\pi_l(.),s))_l$ of meromorphic functions, is bounded in  $0<\Re s<1/2$. 
\item The residus at the $s\in]0,1/2[$  corresponding to small eigenvalues tend to zero.
\item The  expression of the scalar product $\<\frac{1}{l^{1-s}}E_l(z,1-s), \frac{1}{l^{\bar s}}E_l(z,\bar{s}) \>$.
\end{enumerate}
\end{lemma}

\begin{proof}
To show  point (1), the method used for $\Re s>1/2$, the method in \cite{preprint}, p. 367-370,  applies directly.\\
For point (2), it is necessary to emphasize on the dependence on $l$ of the eigenvalues. Let's write $s_j^{\pm}(l)=\frac{1}{2}\pm r_j(l)$ whith $0<r_j(l)\leq 1/2$.

Let $\lambda_j$ $0\leq \lambda_j<1/4$ the small eigenvalues on $S_0$. We know (see for example \cite{coco}) that for $l$ sufficiently small, for $0\leq j \leq M$, $\lambda_j(l)\to \lambda_j$, equivalently $r_j(l) \to r_j$.

 We use \cite{kramer} Theorem 2.
\begin{eqnarray*}
\Res_{s_j^+(l)}(\frac{1}{l^s}\Gamma(s/2)^2\Gamma(s-1/2)^{-1}E_l(z,s))&=&\frac{1}{l^{s_j^+(l)}}
\Gamma(s_j^+/2(l))^2\Gamma(r_j(l))^{-1}\Res_{s_j^+(l)}E_l(z,s)\\
 &=&2\sqrt{\pi}\psi_j(z)\int_{c_l}\psi_j(z)\, ds(z)\, ,
\end{eqnarray*}
which can represent, depending on the multiplicity of $\lambda_j(l)$, a finite sum.

We have a similar expression for the residu at $s_j^-(l)$. We deduce the relation
$$\Res_{s_j^-(l)}(\frac{1}{l^s}E_l(z,s))=l^{s_j^+(l) -s_j^-(l)}\Gamma(s_j^-(l)/2)^{-2}\Gamma(s_j^+(l)/2)^2\Gamma(r_j(l))^{-1}\Gamma(-r_j(l))\Res_{s_j^+(l)}(\frac{1}{l^s}E_l(z,s))\, ;$$
but we have  for $0\leq j\leq M$,  that $\Res_{s_j^+(l)}(\frac{1}{l^s}E_l(z,s)) \to \Res_{s_j^+}(E_\infty +E_0)$, then the result.\\
For point (3), let's write $s=\frac{1}{2}+r$ with $\Re r>1/2$. We compute
\begin{eqnarray*}
& &\<\frac{1}{l^{1-s}}E_l(z,1-s), \frac{1}{l^{\bar s}}E_l(z,\bar{s}) \>=\frac{1}{l}\sum_j a_{j,l}\left(\frac{1}{2}-r\right)  \overline{a_{j,l}\left(\frac{1}{2}+\bar{r}\right)}
=\frac{\pi}{l\Gamma^2(\frac{1/2+r}{2})\Gamma^2(\frac{1/2-r}{2})}\times\\
& &\sum_j \left| \int_{c_l}\psi_j(z) ds(z) \right|^2\, 
\Gamma\left(\frac{-r+ir_j}{2}\right)\Gamma\left(\frac{-r-ir_j}{2}\right)\Gamma\left(\frac{r-ir_j}{2}\right)\Gamma\left(\frac{r+ir_j}{2}\right)\, .
\end{eqnarray*}
We use know the relation $-w\Gamma(w)\Gamma(-w)=\frac{\pi}{\sin(\pi w)}$ to find

$\<\frac{1}{l^{1-s}}E_l(z,1-s), \frac{1}{l^{\bar s}}E_l(z,\bar{s}) \>=\frac{\pi^3}{l\Gamma^2(\frac{1/2+r}{2})\Gamma^2(\frac{1/2-r}{2})}
\sum_j \left| \int_{c_l}\psi_j(z) ds(z) \right|^2\frac{2}{r^2+r_j^2}\frac{1}{\cos(\pi r)+\cos(\pi ir_j)}\, ;$

and we recognize 
$$\left|\left|\int_{c_l} G_s^l(z,w)ds(z)\right|\right|^2_{L^2}=\sum_j\frac{| \int_{c_l}\psi_j(z) ds(z) |^2}{r^2+r_j^2}\, .$$

\end{proof}

We also  use \cite{fayreine}, Formula (99) that connects in particular  for any fuchsian group the hyperbolic Eisenstein series to the following  integral of the kernel of the resolvent
$I_l=\int_1^{e^l} G^l_s(z, iy')d \ln y'\, $.

Because this result will be important for the following sections, we write
\begin{lemma}\label{t}
For {\it any Fuchsian group} $\Gamma_l$, for $s$,  $\Re s>-1/2$, not a pole of  $G^l_s$ and $z$ not in the standard collar, we have
\begin{equation*}
\int_1^{e^l} G^l_s(z, iy')d \ln y'=-\frac{4^{s-1}}{\pi}\frac{\Gamma^4(s)}{\Gamma^2(2s)}N^0_{s}(\frac{\pi}{2})E_l(z,s)+H(z,s)
\end{equation*}
with $H(z,s)=O(1)E_l(z,s+2)$, $O(1)$ bounded independentely of $\Gamma_l$ and of $z$ in $\pi_l(C)$, $C$ a compact of $S_0$.
\end{lemma}
\begin{proof}
 More precisely,  \cite{fayreine}, Formula (99) gives, for $z$ not in the standard collar 
\begin{eqnarray*}
\int_1^{e^l} G^l_s(z, iy')d \ln y'&=&-\frac{4^{s-1}}{\pi}\frac{\Gamma^4(s)}{\Gamma^2(2s)}N^0_{s}(\frac{\pi}{2})e^{\frac{\pi i}{2}s}\sum_{\<\sigma_l\>\backslash \Gamma_l} (\sin \arg\gamma z)^s e^{\mp i  s(\arg \gamma z-\frac{\pi}{2})} \\
&  &.F(s,s,2s;\frac{2}{1\mp i\cot \arg\gamma z})\, .
\end{eqnarray*}
where the $\pm$ is chosen as $\mp(\arg\gamma z-\frac{\pi}{2})>0$. We can then verify that

$$  e^{\mp i  s(\arg \gamma z-\frac{\pi}{2})}=\left\{\begin{array}{cc} e^{-i s\pi/2}(\cos \arg\gamma z +i\sin\arg\gamma z)^s & \text{ if }\arg\gamma z<\pi/2\\ e^{-i s\pi/2}(-\cos \arg\gamma z +i\sin\arg\gamma z)^s & \text{ if }\arg\gamma z\geq \pi/2\end{array}\right .$$
Let $C$ be a compact of $S_0$, $\pi_l(C)\subset S_l\backslash C_l$.

$\forall \gamma\in \<\sigma_l\>\backslash \Gamma_l$, if $\arg\gamma z<\pi/2$, then $0<\arg\gamma z\leq l$;
  if $\arg\gamma z\geq \pi/2$, $0<\pi -\arg\gamma z\leq l$. In every case $0<\sin\arg\gamma z\leq \sin l$.

Moreover
\begin{eqnarray*}
\cos \arg\gamma z +i\sin\arg\gamma z&=&1+2\sin\frac{\arg \gamma z}{2}(-\sin\frac{\arg \gamma z}{2}+i\cos\frac{\arg \gamma z}{2})\\
\cos (\pi -\arg\gamma z) +i\sin(\pi -\arg\gamma z)&=&1+2\sin\frac{\theta}{2}(-\sin\frac{\theta}{2}+i\cos\frac{\theta}{2})\, \text{ with } \theta=\pi -\arg\gamma z\, .
\end{eqnarray*}
Now, if we  have $0<\alpha\leq l_0$, $\cos( l_0/2)\leq \cos(\alpha/2) <1$ and $\sin \frac{\alpha}{2}=O(\sin \alpha)$. We deduce that,  in the case $\arg\gamma z<\pi/2$, $\cos \arg\gamma z +i\sin\arg\gamma z=1+O(\sin\arg\gamma z/2)=1+O(\sin\arg\gamma z)$
and, in the case $\arg\gamma z\geq \pi/2$,  with always $\theta=\pi -\arg\gamma z$,  $\cos (\pi -\arg\gamma z) +i\sin(\pi -\arg\gamma z)=1+O(\sin\theta/2)= 1+O(\sin\theta)=1+O(\sin\arg\gamma z)$.

Then   $e^{\frac{\pi i}{2}s} e^{\mp i  s(\arg \gamma z-\frac{\pi}{2})}=(1+O(1)\sin\arg\gamma z)^s=1+O(1)\sin\arg\gamma z$, $O(1)$, $s-$holomorphic, bounded independently of $\Gamma_l$  and $z$ in $\pi_l(C)$, $C$ compact of $S_0$, for all $s$. More precisely 
\begin{equation*}
e^{\frac{\pi i}{2}s} e^{\mp i  s(\arg \gamma z-\frac{\pi}{2})}=1+2s\sin\frac{\theta}{2}(-\sin\frac{\theta}{2}+i\cos\frac{\theta}{2})+ O((\sin\arg\gamma z)^2), 
\end{equation*}
\begin{equation}\label{the}
  \text{ with } \theta=\pi -\arg\gamma z \quad \text{if  } \arg\gamma z\geq \pi/2,\quad  \theta=\arg\gamma z \quad \text{if not}.
\end{equation}

\medskip
It is well-known (see for example, \cite{h2}, p. 603) that $F(a,b;c;Z)/\Gamma(c)$ defines a holomorphic function for $(a,b,c,z)\in\CC\times \CC\times \CC\times (\CC-[1,\infty[)$.
Then $F(s,s,2s;Z)/\Gamma(2s)$ is $s-$holomorphic and  the eventual poles of $F(s,s,2s;Z)$ are in $\{n/2, n \in \ZZ, n \leq 0\}$. Moreover we easily verify that the set of poles is $\{n/2, n \in \ZZ, n < 0\}$.

 In the same way
$\frac{2}{1\mp i\cot\arg\gamma z}=2\frac{\sin\arg\gamma z}{\sin\arg\gamma z \mp i\cos\arg\gamma z}$ approaches $0$  and 
$$F(s,s,2s;\frac{2}{1\mp i\cot arg\gamma z})=1+\frac{s}{2}\left(\frac{2\sin\arg\gamma z}{\sin\arg\gamma z \mp i\cos\arg\gamma z}\right)+O((\sin\arg\gamma z)^2)$$
with a uniform  $O$-term for $s$ bounded, $\Re s>-1/2$, bounded independently of $\Gamma_l$ and $z$ in $\pi_l(C)$, $C$ compact of $S_0$.\\
Calculating the $s$ coefficient in the product 
$e^{\frac{\pi i}{2}s} e^{\mp i  s(\arg \gamma z-\frac{\pi}{2})} \, F(s,s,2s;\frac{2}{1\mp i\cot arg\gamma z})$, we obtain, with the same definition of $\theta$ as in (\ref{the}): $2(\sin \frac{\theta}{2})^2 e^{i\theta}$.
Then $e^{\frac{\pi i}{2}s} e^{\mp i  s(\arg \gamma z-\frac{\pi}{2})} \, F(s,s,2s;\frac{2}{1\mp i\cot arg\gamma z})= 1+ O(1)(\sin \arg\gamma z)^2$ with a uniform  $O(1)$-term for $s$ bounded, $\Re s>-1/2$, bounded independently of $\Gamma_l$ and $z$ in $\pi_l(C)$, $C$ compact of $S_0$.\\

\medskip

Then, for $\Re s >1$

%

\begin{equation}
I_l +\frac{4^{s-1}}{\pi}\frac{\Gamma^4(s)}{\Gamma^2(2s)}N^0_{s}(\frac{\pi}{2})E_l(z,s)=  -\frac{4^{s-1}}{\pi}\frac{\Gamma^4(s)}{\Gamma^2(2s)}N^0_{s}(\frac{\pi}{2})  \sum_{\Gamma_l\backslash \Gamma} (\sin \arg\gamma z)^s O(1)(\sin \arg\gamma z)^2\, .
\end{equation}

The right member of the last equality is a convergent and analytic series for $\Re s>-1$.
Then by analytic continuation we can deduce that for $\Re s >-1/2$, $s$ not a pole of  $G^l_s$
\begin{eqnarray*}
I_l &=&-\frac{4^{s-1}}{\pi}\frac{\Gamma^4(s)}{\Gamma^2(2s)}N^0_{s}(\frac{\pi}{2})E_l(z,s)  -\frac{4^{s-1}}{\pi}\frac{\Gamma^4(s)}{\Gamma^2(2s)}N^0_{s}(\frac{\pi}{2})  \sum_{\Gamma_l\backslash \Gamma} (\sin \arg\gamma z)^s O(1)(\sin \arg\gamma z)^2\\
&=&-\frac{4^{s-1}}{\pi}\frac{\Gamma^4(s)}{\Gamma^2(2s)}N^0_{s}(\frac{\pi}{2})E_l(z,s) + H(z,s)\, .
\end{eqnarray*}
%
where $H(z,s)=0(E_l(z,s+2))$.
\end{proof}

\begin{rem}
This Lemma \ref{t} will be improved in the infinite volume case.
\end{rem}
We are now able to prove  Proposition \ref{3.1}.
\begin{proof}
 In the compact case we have $G_s=G_{1-s}$, so we obtain, from Lemma \ref{t}, in particular  for $0<\Re s <1/2$,
\begin{equation}\label{compact}
E_l(z,s)=K(s) E_l(z,1-s)+O(E_l(z,3-s))+ O(E_l(z,s+2)),
\end{equation} 
with $K(s)$  depending only on $s$.
Then dividing by $1/l^s$ and using the known  results on the convergence of hyperbolic Eisensteins series (see \cite{preprint}) we have

$\displaystyle\frac{E_l(z,s)}{l^s}=O(l^{1-2s})$, and the result.

\end{proof}

\begin{rem}
In view of Lemma \ref{t}, we can also notice  the formula  in \cite{wave}, Proposition 11.
\end{rem}

When $S_l$ is compact, you have to give a sense to the convergence of $(\frac{1}{l^s}E_l(z,s))_l)$ on a domain that intersect $\Re s=1/2$ because the eigenvalues of $S_l$ cluster at every point of the continuous spectrum $[\frac{1}{4}, +\infty[$ of $S_0$ as $l\to 0$ (see for example \cite{jit}). To remove this problem we multiply the hyperbolic Eisenstein series by the Selberg Zeta function to obtain  holomorphic functions. More precisely, we recall (see for example \cite{h1}, p.72) that in the compact case
\begin{enumerate}
\item  $Z_l(s)$ is actually an entire function;
\item $Z_l(s)$ has "trivial" zeros $s=-k$, $k\geq 1$, with multiplicity $(2g-2)(2k+1)$;
\item $s=0$ is a zero of multiplicity $2g-1$;
\item $s=1$ is a zero of multiplicity 1;
\item the nontrivial zeros of $Z_l(s)$ are located at $\frac{1}{2}\pm i r_j$ of order for $r_j\not = 0$, equal to the dimension of the $\lambda_j$-eigenspace except for $s=1/2$ where $Z_l(s)$ has a zero of multiplicity equal to twice the dimension of the $(1/4)$-eigenspace.
\end{enumerate}
It follows that  for $\Re s >0$, poles of  the hyperbolic Eisenstein series are zeros of the Selberg Zeta function and the multiplicity of a pole of $E_l(z,s)$ is less or equal the multiplicity of the corresponding zero of $Z_l(s)$, then 
\begin{lemma}
The product $Z_l(s) E_l(z,s)$ is a  holomorphic function in $\Re s >0$.
\end{lemma}

\begin{cor}
The family $E^*_l(z,s)=\frac{Z_l(s)}{z_l(s)}\frac{1}{l^s}E_l(z,s)$ cannot converge for $0<\Re s <1/2$ as $l\to 0$.
\end{cor}
\begin{proof}
To see this, from Proposition \ref{3.1} and the preceding study (\ref{compact}) , $\frac{1}{l^s}E_l(z,s)= O(l^{1-2s})$ on $0<\Re s <1/2$. 
Moreover we know (see \cite{wp}) that  on $0<\Re s <1/2$,
\begin{equation}\label{finite}
\frac{Z_l(s)}{z_l(s)}= \frac{Z_l(1-s)}{z_l(1-s)} O(l^{4s-2})\, . 
\end{equation}
Then $\frac{Z_l(s)}{z_l(s)}\frac{1}{l^s}E_l(z,s)=O(l^{2s-1})$. Simply observe that for $0<\Re s <1/2$ and $l\to 0$ the factor $l^{2s-1}$ is divergent.
\end{proof}

\section{The non-compact finite volume case}

Here again we show that the meromorphic family, $( \frac{1}{l^s}E_l(z,s))$ tends to zero with $l$  on $0<\Re s <1/2$. For this purpose we will follow the same method as the preceding case.

Doing this we will find again the convergence of the scattering matrix in the finite volume (non compact) case  on every compact of $\Re s>1/2\backslash \{s_0,...,s_{M}\}$, where these $s_j$ correspond to the small eigenvalues of the limit surface and this will enable us to lift a condition of Schulze \cite{schulze}, p.122.
%

In this case $S_l$,  being a noncompact finite-volume  hyperbolic surface,  has  finitely many cusps and Eisenstein series associated to them. The Laplacian on $\Delta_l$ has absolutely continuous spectrum $[1/4,\infty)$ with multiplicity the number of cusps of $S_l$. The discrete spectrum consists of finitely many eigenvalues in $[0,1/4)$: $\lambda_n=s_n(1-s_n)$,  where $s_n=\frac{1}{2} +i r_n$, $\tilde{s_n}=\frac{1}{2} -i r_n$, $s_n\in [\frac{1}{2},1]\cup [\frac{1}{2},\frac{1}{2}+i\infty)$. There are examples with infinitely many embedded eigenvalues in $[1/4,\infty)$.
The theory of Eisenstein series and their analytic continuation furnishes the continuous spectrum.
 Without loss of generality we can assume that there is only one cusp $p$ with $E_p^l(z,s)$ the associated Eisentein series.
 The constant term in the Fourier expansion of the Eisenstein series, $\phi_l(s)$ is meromorphic in $\mathbb C$. Its only poles in $\Re s \geq 1/2$ are in $(\frac{1}{2},1]$; these poles $\{\rho_0,\rho_1,...,\rho_{M_e}^l\}$  are finite and simple. They coincide with the poles of $E_p^l(z,s)$ in $\Re s \geq 1/2$.
The residues at these poles furnish solutions to the eigenvalue problem, called the residual spectrum of $S_l$. The poles of $\phi_l(s)$ in $\Re s <1/2$ yield resonances .
\begin{eqnarray*}
L^2(S_l)&=&{\mathcal H}_{res}+ {\mathcal H}_{cusp}+{\mathcal H}_{cont}\\
&=&\sum_{0\leq n\leq M_e^l}\<\psi_n\> +\sum_{n>M_e^l}\<\phi_n\>+{\mathcal H}_{cont}\\
\Delta_l \psi_n +\rho_n(1-\rho_n)\psi_n=0\, ,& \quad &\Delta_l \phi_n +c_n(1-c_n)\phi_n=0\, .
\end{eqnarray*}

The orthogonal complement in $L^2(S_l)$ of the continuous and residual spectrum is the cuspidal space. It is invariant under $\Delta_l$, and the resolvent is compact when restricted to the cuspidal space. 

 The resolvent  $R_l(s):L^2(S_l)\rightarrow L^2(S_l)$ has a meromorphic continuation to the whole complex plane. More precisely (\cite{h2}, p.250), the singularities of the kernel of the resolvent correspond to $\{s_n,\tilde{s_n}\}\cup\{1/2\}\cup \{\text{ the poles of } \phi_l \}$. $\{s_n,\tilde{s_n}\}$ are simple poles  with the  possible exception of the pole at  $s=1/2$.

The kernel of the resolvent is given for $\Re s >1/2$  by
$G_s^l(z,w)=\sum_{n=0}^\infty \frac{\psi_n(z)\overline{\psi_n(w)}}{\lambda-\lambda_n} +\frac{1}{2\pi}\int_0^{+\infty} \frac{E_p(z,\frac{1}{2} +ir)E_p(w,\frac{1}{2} -ir)}{s(1-s)-(\frac{1}{4} +r^2)}\, dr$.

We return now to the study of the hyperbolic Eisenstein series.
The hyperbolic Eisenstein series admits  the spectral expansion
$$E_l(z,s)=\sum_{j=0}^\infty a_{j,l}(s) \psi_j(z) +\frac{1}{4\pi}\int_{-\infty}^{+\infty} a_{1/2 +ir,l}(s)E_p^l(z, 1/2+ir)\, dr\, .$$
The coefficient $a_{j,l}$ is given by the formula
\begin{eqnarray*}
a_{j,l}(s)&=&\sqrt{\pi}\frac{\Gamma((s-\frac{1}{2}+ir_j)/2)\Gamma((s-\frac{1}{2}-ir_j)/2)}{\Gamma(s/2)^2}\int_{c_l} \psi_j(z)\, ds(z)\\
a_{1/2 +ir,l}(s)&=&\sqrt{\pi}\frac{\Gamma((s-\frac{1}{2}+ir)/2)\Gamma((s-\frac{1}{2}-ir)/2)}{\Gamma(s/2)^2}\int_{c_l} E_p^l(z,1/2+ir)\, ds(z)
\end{eqnarray*}
The eventual poles are located at the points
\begin{enumerate}
\item $s=1/2 \pm i r_j-2n, n\in \NN$ and they are simple except in the eventual case $s_j=1/2$, of order 2.
\item $s=\rho -2n$, where $n\in\NN$ and $t=\rho$ is a pole of the Eisenstein series $E_p^l(z,t)$ with $\Re(\rho)<1/2$.
\end{enumerate}

\bigskip

We will prove that
\begin{prop}\label{4.1}
 $\frac{1}{l^s}E_l(z,s)= O(l^{1-2s})$ on $0<\Re s <1/2$, such that $s(1-s)$ belongs to the resolvent set of the Laplacian on $S_0$.
\end{prop}
Before we need
\begin{lemma}
On $\Re s<1/2$, as a family of meromorphic function, $\phi_l(s)$  tends  to $\phi_0(s)$
\end{lemma}
\begin{proof}
We know the convergence of  the family of $s$-meromorphic functions $( E_p^l(\pi_l(w),s))_l$  to $E_p^0(w,s)$  uniformly  on any compact subset of $S_0$ (see \cite{o4}, \cite{jorg} for $\Re s >1$, \cite{schulze} and \cite{hc} for $\Re s>1/2$).

As $\int_0^1 E_p^l(z,s)\, dx=y^s+\phi_l(s)y^{1-s}$ we deduce the convergence of $(\phi_l(s))$ to $\phi_0(s)$ on every compact of $\Re s>1/2\backslash \{s_0,...,s_{M_e}\}$ (see also \cite{h2}, Proposition 12.5). By Hurwitz's theorem, the zeros of $\phi_l$ converge to those of $\phi_0$. Then, for $\Re s>1/2$ $\displaystyle\left( \frac{1}{\phi_l(s)}\right)$ is a meromorphic family converging to  $\displaystyle \frac{1}{\phi_0(s)}$.

Now for $\Re s<1/2$, $\Re( 1-s)>1/2$ and for $l\geq 0$, $\phi_l(s)=1/\phi_l(1-s)$, so the conclusion.
\end{proof}
\begin{cor}
  $E_p^l(z,s)$ for $\Re s<1/2$,  tends to $E_p^0(z,s)$.
\end{cor}

\begin{proof} of Proposition \ref{4.1}.
 In the finite volume non compact case, the resolvent is no more symetric and we have the relation
$$G_s(z,iy')=G_{1-s}(z,iy')-\frac{1}{2s-1} E_p^l(z,s)E_p^l(iy',1-s)\, .$$

Hence we have for $0<\Re s <1/2$
\begin{equation}\label{rt}
\int_1^{e^l} G_s(z, iy',l)d \ln y'=\int_1^{e^l} G_{1-s}(z, iy',l)d \ln y' -\frac{1}{2s-1}E_p^l(z,s)\int_1^{e^l} E_p^l(iy',1-s)d \ln y'\, .
\end{equation}
For $0<\Re s<1/2$, 
$\displaystyle\int_1^{e^l} E_p^l(iy',1-s)d \ln y'=lO(1) $, see \cite{schulze}  Definition 19, 20 pp. 138-139 and Theorem 30 p. 147.

Dividing (\ref{rt}) by $l^s$ and using, like in the compact case, Lemma (\ref{t}),
the  conclusion follows from the preceding study.
\end{proof}

As in the compact case we multiply the hyperbolic Eisenstein series by the Selberg Zeta function to obtain  meromorphic functions,    holomorphic on $\Re s>0\backslash \{1-\rho_k, 0\leq k\leq M_e^l, 1/2\}$. More precisely, we recall (see for example \cite{h2}, p.498) that in the finite volume non compact case
\begin{enumerate}
\item  $Z_l(s)$ is meromorphic on the $s$-plane;
\item $Z_l(s)$ has "trivial" zeros $s=-k$, $k\geq 1$.
\item the nontrivial zeros of $Z_l(s)$ are located at points $s_j$, on the real line $\Re s=1/2$, $s_j=\frac{1}{2}\pm i r_j$ of multiplicity, for $r_j\not =0$, equal to the dimension of the $\lambda_j$-eigenspace. 
\item for $s=1/2$  $Z_l(s)$ has a zero  or pole of multiplicity equal to twice the dimension of the $(1/4)$-eigenspace minus the number of inequivalent cusps $\frac{1}{2}\text{Tr}[I -\phi_l(\frac{1}{2})]$;
\item  the nontrivial zeros of $Z_l(s)$ are located for $\Re s > 1/2$ at points $s_j$ and $\rho_k$, $0\leq k\leq M^l_e$ in $]\frac{1}{2},1]$ of multiplicity equal to the dimension of the corresponding eigenvalue;
\item it has additional non-trivial zeros in $\Re s<1/2$
 at the poles of $\phi_l(s)$ with the same multiplicity.
\item $Z_l(s)$ has poles at $s=-l+1/2$, $l\geq 1$, of order 1. 
\end{enumerate}

\medskip

In order to write the functional equation, we need the entire function $G_\infty$  which is defined in terms of the Barnes G-function (\cite{borth},p. 44)
\begin{eqnarray}\label{Barnes}
G(s+1)&=&(2\pi)^{s/2}e^{-s/2-(\gamma+1)s^2/2}\displaystyle\Pi_{k=1}^\infty \left(1+\frac{s}{k}\right)^ke^{-s+s^2/(2k)}\, ,\\
G_\infty(s)&=&(2\pi)^{-s}\Gamma(s)G(s)^2
\end{eqnarray}
We then have a relation between $Z_l(s)$ and $Z_l(1-s)$ in the finite volume and  non compact case (see for example \cite{borth}, p. 45, also \cite{h2}, formula (5.7) p.499) 
$$\frac{Z_l(s)}{Z_l(1-s)}=\left(\frac{G_\infty(s)}{G_\infty(1-s)}\right)^{-\chi(S_l)}\, .$$

It allows us to answer the remark in the same manner as in the compact case, because we have again $\frac{Z_l(s)}{z_l(s)}= \frac{Z_l(1-s)}{z_l(1-s)} O(l^{4s-2})$.
\begin{lemma}
The quotient $\displaystyle\frac{Z_l(s)}{z_l(s)}$ cannot converge for $\Re s<\frac{1}{2}$ as $l\to 0$.

\end{lemma}

 Then, from Proposition (\ref{4.1}), $\frac{Z_l(s)}{z_l(s)}\frac{1}{l^s}E_l(z,s)=O(l^{2s-1})$  and
\begin{cor}
The meromorphic family, $E^*_l(z,s)=\frac{Z_l(s)}{z_l(s)}\frac{1}{l^s}E_l(z,s)$ cannot converge for $0<\Re s <1/2$ as $l\to 0$.
\end{cor}

\begin{rem}
In the finite volume case, looking at the particular point s = 1/2, we could already conclude
that the convergence of the family $\frac{Z_l(s)}{z_l(s)}  \frac{1}{l^s}E_{l}(z,s)$  to the desired series is not true in general:
in the case $Z_0$ has a pole at $s=1/2$ which occurs when twice the dimension of the 1/4 eigenvalue space minus $1/2 \text{ Tr}[ I-\Phi(1/2)] <0$.
\end{rem}

\section{The infinite volume case}

For the remainder of this section, assume, for simplicity, that $\Gamma$ is a Fuchsian group of the second kind whose fundamental domain $\mathcal D$ has no cusps and only  one free side $\overline{\mathcal D}\cap \RR=[1,e^l]$ corresponding to the cyclic subgroup  generated by $\sigma_l(z)=e^l z$. The corresponding funnel is denoted by $F_l$.\\
A collar $\mathcal C$ for $\sigma_l$ is the cylinder
$${\mathcal C}=\{z=re^{i\theta}, 1\leq r \leq e^l, cl<\theta<\pi-cl\}/ \<z\to e^l z\>$$
for a choice of constant $c$. For $l$ small the standard collar, $c=1$, embeds in $S_l$ to give a neighborhood of $\sigma_l$, and has: area $\approx 2$, width $\approx 2\log\frac{1}{l}$, and each boundary of length $\approx 1$. 

Here also, we adapt the notations in Borthwick \cite{borth}.

The Poisson operator is the map 

$$E_{S_l}(s): C^{\infty}(\partial\overline{S_l})\longrightarrow C^{\infty}(S_l)$$

which maps $f\in  C^{\infty}(\partial\overline{F_l})$ to

$$E_F(s)f(\rho,t)=\int_0^l E^{f_l}(s; \rho,t,t')f(t')\, dt'\, ;$$
where 

$$E^{f_l}(s,z,t')=(1-2s)\lim_{\rho '\to 0} {\rho'}^{-s}G_s(z,z')\, ;$$
and  $S_l(s)$ is the corresponding scattering matrix. The funnel scattering operator $S_{F_l}$ is defined in the same way.

The relative scattering determinant $\tau_{S_l}(s)$ is a Fredholm determinant. $S^l_{rel}(s)-I$ is a smoothing operator, with $S^l_{rel}(s)=S_{F_l}(1-s)S_{l}(s)$.

$$\tau_{S_l}(s)=\det S^l_{rel}(s)=\Pi_{k=1}^\infty \left(1+\lambda_k(S^l_{rel}(s)-I)\right)$$

where $\lambda_k(S^l_{rel}(s)-I)$ are the eigenvalues of $A_l:=S^l_{rel}(s)-I$.

\begin{thm}\label{clef}
For $Re s>1/2$ we have
$$\tau_{S_l}(s)\to \frac{1}{2(\sin\pi/2 s)^2}, l\to 0, .$$
\end{thm}

We recall the following eigenfunction expansion (see \cite{fayreine}  formula (98)')
\begin{equation}
 G_s^l(z,z')=\frac{1}{1-2s}\sum_{n\in \ZZ}F_n^l(z,s)N^n_s(\phi')|z'|^{\bar{n}}
\end{equation}
when $\arg z> \arg z'=\phi'$ and $z'$ in the collar; here
\begin{equation}\label{7}
F_n^l(z,s)=\frac{4^{s-1}}{\pi}\frac{\Gamma(s)^2\Gamma(s+\bar{n})\Gamma(s-\bar{n})}{\Gamma(2s)\Gamma(2s-1)}\frac{1}{l}\sum_{\Gamma_l\backslash \Gamma}N_s^{-n}(\pi -arg\gamma z)|\gamma z|^{-\bar{n}}\, ,
\end{equation}
with $\Re s >1$ to ensure convergence of the series.

We recall know the standard Fourier expansion for the resolvent of the limit surface $S_0=\Gamma\backslash H$, $\Gamma$ containing a parabolic subgroup $\Gamma_\infty$ generated by the transformation $z\mapsto z+1$ (see \cite{fayreine}  formula (73), \cite{h2}, p. 42) .
For $y'>y$, the resolvent has an eigenfunction expansion
\begin{equation}\label{li}
G_s(z,z')=\frac{1}{1-2s}\sum_n F_n(z,s) W_{s-\frac{1}{2}}(4\pi|n|y')e^{2\pi i nx'}
\end{equation}
where the $n^{th}$ Fourier coefficient is given by
\begin{equation}
F_n(z,s)=\frac{1}{4\pi}\frac{\Gamma(s)}{|n|\Gamma(2s-1)}\sum_{\gamma\in \Gamma_\infty\backslash \Gamma}M_{s-\frac{1}{2}}(4\pi|n|\Im \gamma z)e^{-2\pi i n \Re \gamma z}\, ,
\end{equation}
with the convention that the constant term in (\ref{li}) is $y'^{1-s} E_\infty(z,s)$.

\smallskip

We begin with estimating the Fourier coefficients of the resolvent of the laplacian on $S_l$. 
We are now going to establish the convergence of the Fourier coefficients for $\Re s >1$.
Following straight fully the method to study the degeneration of hyperbolic Eisenstein series like in \cite{preprint}, or \cite{pubprims},
 we find that for $\Re s>1$
\begin{prop}
\begin{enumerate}
\item The family $(l^{1-s}F_0^l(., s))_l$ converges uniformly on compact subsets of $S_0$ and on compact subsets of $\Re s >1$ to $ \frac{4^{s-1}}{\pi}\frac{\Gamma(s)^4}{\Gamma(2s)\Gamma(2s-1)}\frac{1}{l^{s+1}} E_\infty(.,s)$.
\item For $n\not =0$, the family $(l^{s} e^{\frac{2|n|\pi^2}{l}} F_n^l(.,s))_l$ converges uniformly on compact subsets of $S_0$ and on compact subsets of $\Re s >1$ to

\end{enumerate}

\end{prop}

\begin{proof}
To study the left side of the collar for $\sigma_l$, use the change of variables $l\zeta=-\log(-z)$, with the principal branch; then 
\begin{equation}
\arg z=\pi -l\Im \zeta, \quad |z|=\exp(-l\Re \zeta).
\end{equation}
We briefly recall the method  and refer to \cite{preprint}, \cite{pubprims}, for more details.

Conjugate $\Gamma_l$ by the map $\zeta$ to obtain ${\tilde\Gamma}_l$ acting on ${\mathcal S}_l=\{\zeta, 0<\Im \zeta<\pi/l\}$.

Given $\epsilon>0$, denote by $G_\epsilon$ the set of cosets and
representatives for $\<z\mapsto z+1\>\backslash \Gamma$ such that $\sup
\Im A(\mathcal F)<\epsilon$  for $[A]\not \in G_\epsilon$ and let $R_l$ be
the corresponding cosets of $\<z\mapsto z+1\>\backslash {\tilde\Gamma}_l$
with the corresponding representatives. 
The set $G_\epsilon$ is finite and so is $R_l$.  We have
$$\sum_{\<\sigma_l\>\backslash \Gamma_l}N_s^0(\pi -arg\gamma z)=\sum_{\<z\to z+1\>\backslash{\tilde\Gamma}_l- R_l}N_s^0(l\Im\delta \zeta)
 + \sum_{R_l}N_s^0(l\Im\delta \zeta)\, .$$
By  formula (\ref{def1}), as $l\to 0$, $N_s^0(l\Im\delta \zeta)\sim l^s \Im^s\delta\zeta$; thus for $\frac{1}{l^s}\sum_{\<\sigma_l\>\backslash \Gamma_l}N_s^0(\pi -arg\gamma z)$, in the last equality the second  sum is uniformly close to the corresponding term for $E_\infty(.,s)$.

 For the first sum recall that
$$N_s^0(\theta)=(\sin\theta)^s\exp(i\theta s)F\left(s,s,2s;-2i\sin\theta e^{i\theta}\right)\, .$$


With $\theta=\pi -\arg \gamma z=l\Im\delta \zeta$,
and that, for $\delta \in \<z\to z+1\>\backslash{\tilde\Gamma}_l- R_l$, $\Im \delta \zeta\leq 2\epsilon$, then $l\Im \delta \zeta \to 0$ uniformly in $\delta \in  \<z\to z+1\>\backslash{\tilde\Gamma}_l- R_l$, we obtain that
  $$\frac{1}{ l^s}\sum_{\<z\to z+1\>\backslash{\tilde\Gamma}_l- R_l}N_s^0(l\Im\delta \zeta) =O(\epsilon^{s-1})$$

see also \cite{jorg}. Now



$$F_0^l(z,s)=\frac{4^{s-1}}{\pi}\frac{\Gamma(s)^4}{\Gamma(2s)\Gamma(2s-1)}\frac{1}{l}\sum_{\Gamma_l\backslash \Gamma}N_s^0(\pi -arg\gamma z)$$
and by the preceding change of variables
$(l^{s-1}F_0^l(.,s))_l$  tend to $$\frac{4^{s-1}}{\pi}\frac{\Gamma(s)^4}{\Gamma(2s)\Gamma(2s-1)}\sum_\delta \Im (\delta \zeta) ^s\, $$
uniformly on compact subsets of $S_0$ and on compact subsets of $\Re s >1$.

In a similar way
$$ N_s^n(\theta)=(\sin\theta)^s\exp\left[i\theta\left(s+\bar{n}\right)\right]F\left(s+\bar{n},s,2s;-2i\sin\theta e^{i\theta}\right)  $$
\end{proof}

We have the formula

$$\frac{1}{1-2s}F_n^l(z,s)N_s^n(\phi')=\frac{1}{l}\int_1^{e^l}G_s^l(z,z')|z'|^{-\bar{n}}\, d\ln|z'|$$
with the change of variable $t'=\ln|z'|$  we obtain

$$\frac{1}{1-2s}F_n^l(z,s)N_s^n(\phi')=\frac{1}{l}\int_0^l G_s^l(z,z')e^{-2i\pi n t'/l}\, d t'\, .$$
Using the change of variables in Schulze $t'=-lx'$ we write

$$\frac{1}{l}\int_0^l G_s^l(z,z')e^{-2i\pi n t'/l}\, d t'=\int^0_{-1} G_s^l(z,z')e^{2i\pi n x'}\, d x'\, $$

With the convergence of the resolvent (see Schulze prop 6, p.126), we have the convergence of $\int_0^1 G_s^l(z,z')e^{-2i\pi n\Re w'}\, d\Re w'\, $ to 
\begin{enumerate}
\item for $n\not =0$, $F_n(w,s) W_{0, s-\frac{1}{2}}(4\pi|n|y')$
\item for $n=0$, $\frac{1}{1-2s}\sum_{\gamma} (\Im \gamma w)^s =\frac{1}{1-2s} E_\infty(w,s) $

\end{enumerate}

We can then write $\phi'=\pi -la$ and then use

\begin{enumerate}
\item for $n\not =0$, $N_s^n(\pi -la)=O( l^s e^\frac{2|n|\pi^2}{l})$
\item for $n=0$, $N_s^n(\pi -la)=O( l^{1-s})$

\end{enumerate}

We will use
\begin{rem}
$\frac{\pi 4^{1-s}\Gamma(2s)\Gamma(2s-1)}{ \Gamma(s)^2 }=\pi \frac{4^s}{2}\frac{\Gamma(s-\frac{1}{2})}{\Gamma(\frac{1}{2}-s)\cos \pi s} $
\end{rem}
\begin{proof}
We use \cite{fayreine}, formula (89).\\
For $n=0$, we have
 $$ N_s^0(\pi -l a)=\frac{\pi \Gamma(s-\frac{1}{2})}{2\Gamma(\frac{1}{2} -s)\cos \pi s \Gamma(s)^2}4^s N_{1-s}^0(l a) +\frac{1}{\cos \pi s} N_s^0(l a)\, .$$
As $\Re s>1/2$, we deduce that, as $l\to 0$, $ N_s^0(\pi -l a)\sim\frac{\pi \Gamma(s-\frac{1}{2})}{2\Gamma(\frac{1}{2} -s)\cos \pi s \Gamma(s)^2}4^s a^s l^s$

For $n\not=0$, we have
 $$ N_s^n(\pi -l a)=\frac{\pi 4^{1-s}\Gamma(2s)\Gamma(2s-1)}{\Gamma(s+\bar{n})\Gamma(s-\bar{n}) \Gamma(s)^2 } N_{1-s}^n(l a) +\frac{\cos(\pi\bar{n})}{\cos \pi s} N_s^n(l a)\, .$$
The following behaviors 
\begin{eqnarray*}
 \cos(\pi\bar{n})&\sim& \frac{1}{2}e^{\frac{2\pi^2|n|}{l}}\\
 \Gamma(s+\bar{n})\Gamma(s-\bar{n})&\sim& 2\pi e^{-2s}\left(\frac{2\pi |n|}{l}\right)^{2s-1} e^{-\frac{2\pi^2}{l}|n|}
\end{eqnarray*}
allow to conclude.
\end{proof}
to conclude that
 the families $(l^{1-s}F_0^l(z,s))_l$ and $l^s e^\frac{2\pi^2|n|}{l}F_n^l(z,s))_l$ are convergent.

\bigskip

\smallskip

We then have the Fourier expansion of the generalized Eisenstein series

\begin{prop}\label{r}
$$E^{f_l}(s,z,t')=\sum_{n\in\ZZ} F_n^l(z,s)e^{\frac{2i\pi t'}{l}}$$
with the following asymptotic behavior of the Fourier coefficients, uniformly on compact subsets of $S_0$ and on compact subsets of $\Re s >1/2$  :
\begin{enumerate}
\item $\displaystyle l^{1-s}  F_0^l(.,s)$ tend to $E_\infty(.,s)$, 
\item  for $n\not =0$, $\displaystyle e^\frac{2|n|\pi^2}{l} l^{s} F_n^l(.,s)$
tend to $F_{-n}^0(.,s)$, as $l\to 0$.
\end{enumerate}
\end{prop}

\begin{rem}
\begin{enumerate}
\item For $\Re s>1$, we can refind this result by hand following Fay p.201.
\item For $n=0$, we can find a link between  $F_n^l$ and the hyperbolic Eisenstein series yet mentioned
$$\int_0^l G_s(z,iy')\, dt'=l F_0^l(z,s) N_s^0(\phi')$$ with $\phi '=\pi/2$
$$=-\frac{4^{s-1}}{\pi}\frac{\Gamma(s)^4}{\Gamma^2(2s)} N_s^0(\pi/2) \left[E_{hyp}^l(z,s)+0(E_{hyp}^l(z,s+1)\right]$$

\end{enumerate}

\end{rem}

The coefficients (\ref{7}) have the eigenfunction expansion (see \cite{fayreine}, formula (102)')
\begin{equation}\label{co}
lF_{-n}(z,s)=N^n_{1-s}(\phi)|z|^{\bar{n}}+\sum_{m\in \ZZ}\beta_{m,n}(s)N_s^m(\phi)|z|^{\bar{m}}.
\end{equation}
Notice that $\beta_{m,n}(s)$  is independent of $\phi=\arg z$. To explicit $\beta_{m,n}(s)$ we need to define the following quantities.
For $m, n\in \ZZ$ define the Kloostermann sum of hyperbolic type:
\begin{equation}\label{klooster}
\sigma(m,n,q)=\sum_{\gamma=\left(\begin{array}{cc}a&b\\c&d\end{array}\right)\in\<\sigma_l\>\backslash \Gamma^*/\<\sigma_l\>}
\left|\frac{a^mc^n}{b^ma^n}\right|^{\bar{n}},\, q=\frac{ad}{bc},
\end{equation}
where the summation is taken over all $\gamma\not=\left(\begin{array}{cc}1&0\\0&1\end{array}\right),\, \left(\begin{array}{cc}0&-\frac{1}{c}\\c&0\end{array}\right)\in \Gamma$.  Here $q$ is independent of the class of $\gamma$ in $\<\sigma_l\>\backslash \Gamma^*/\<\sigma_l\>$, and $0<q=1+\frac{1}{bc}<1$ since
$$\gamma(0)=\frac{b}{d},\quad \gamma(\infty)=\frac{a}{c},\quad\gamma^{-1}(0)=\frac{-b}{a},\quad \text{ and }\gamma^{-1}(\infty)=\frac{-d}{c}$$
are all strictly negative. Then for $n\in \ZZ$,
\begin{equation*}
\beta_{m,n}(s)=\frac{4^{s-1}}{\pi}\frac{\Gamma(s)^2\Gamma(s+\bar{n})\Gamma(s-\bar{n})}{\Gamma(2s)\Gamma(2s-1)}\frac{1}{l}
\left\{ \sum_{0<q<1}\sigma(m,-n,q){\mathcal I}_s(m,-n,q)+\delta_m^n \frac{\cos\pi\bar{n}}{\cos \pi s}\right\}\, ,
\end{equation*}
where 
$${\mathcal I}_s(m,n,q)=(1-q)^s q^{\bar{n}-s}\frac{\Gamma(s-\bar{m})\Gamma(s+\bar{m})}{\Gamma(2s)} F(s-\bar{n}, s-\bar{m},2s;\frac{q-1}{q})\, .$$
From the behavior of $N_s^n(\phi)$ we deduce
\begin{equation}
S_{l}(s)f_n=\sum_{m\in \ZZ}\beta_{m,n}(s) f_m
\end{equation}
 where $f_k(t)=e^{\frac{2i\pi k t}{l}}$ is the Fourier basis for the scalar product
$$\<f,g\>=\frac{1}{l} \int_0^l f(t)\bar{g}(t)\, dt\, .$$
The $f_k$, where $f_k(t)=e^{2i\pi k t}$,  are eigenfunctions of $S_{F_l}(s)$,

$$S_{F_l}(s)f_k=\gamma_k(s)f_k ,$$
with eigenvalues
\begin{equation}\label{fun}
\gamma_k(s)=\frac{\Gamma(\frac{1}{2}-s)\Gamma(\frac{s+1+\bar{k}}{2})\Gamma(\frac{s+1-\bar{k}}{2})}{\Gamma(s-\frac{1}{2})\Gamma(\frac{2-s+\bar{k}}{2})\Gamma(\frac{2-s-\bar{k}}{2})}
\end{equation}
where $\bar{k}=\frac{i2k\pi}{l}$.

Then 
$$S^l_{rel}(s) f_n=\sum_{m\in \ZZ}\beta_{m,n}(s) \gamma_m(1-s) f_m\, .$$

$$\beta_{n,n}(s) =\frac{I_n - N_{1-s}^n(\phi)}{N_{s}^n(\phi)}$$
where 
$$I_n=\int_0^l F_{-n}^l(z,s)e^{\frac{-2i\pi n u}{l}}\, du$$ and $u=\ln|z|$.
Denote by $u_n(l)=\beta_{n,n}(s)\gamma_n(1-s) -1$,
\begin{equation}\label{esti}
u_n(l)=\frac{\gamma_n(1-s)I_n-\gamma_n(1-s)N_{1-s}^n(\phi)-N_s^n(\phi)}{N_s^n(\phi)}\, .
\end{equation}
With $\phi=\pi -l\Im \zeta$, let estimate $u_n(l)$ as $l\to 0$.
%
%
%

We verify that
\begin{equation*}
\gamma_n(1-s)N_{1-s}^n(\pi - l\Im\zeta) +N_{s}^n(\pi-l\Im\zeta)=-\left(\gamma_n(1-s)N_{1-s}^n(l\Im \zeta) +N_{s}^n(l\Im \zeta)\right)
\end{equation*}
and recall that
 $$ N_s^n(\pi -l \Im\zeta)=\frac{\pi 4^{1-s}\Gamma(2s)\Gamma(2s-1)}{\Gamma(s+\bar{n})\Gamma(s-\bar{n}) \Gamma(s)^2 } N_{1-s}^n(l \Im\zeta) +\frac{\cos(\pi\bar{n})}{\cos \pi s} N_s^n(l \Im\zeta)\, .$$

For $n=0$,the numerator and denominator in (\ref{esti}) can be expressed as a $l^s O(1) + l^{1-s} O(1)$, so because $\Re s >1/2$, we have only to precise the coefficient of $l^{1-s}$. Using Proposition \ref{r}, we have
$I_0=\int_0^l F_0^l(z,s)\, du = O(l^s)$, thus

$$u_0(l)\sim \frac{\gamma_0(1-s)}{\pi \frac{4^s}{2}\frac{\Gamma(s-\frac{1}{2})}{\Gamma(\frac{1}{2}-s)\cos \pi s \Gamma(s)^2}} = \frac{\cos \pi s}{2(\sin\frac{\pi}{2}s)^2}\, .$$

For $n\not = 0$ we obtain
$$u_n(l)=O\left(e^{-\displaystyle\frac{2\pi^2|n|}{l}}\left|\frac{n}{l}\right|^{1-2s}\right)\, .$$

\bigskip
For $m\not =n$, we are going to establish the asymptotic behavior of 
$$u_{m,n}(l)=\beta_{m,n}(s)\gamma_m(1-s)\, .$$
We have 
$$I_{m,n}=\int_0^l F_{-n}^l(z,s)e^{\frac{-2i\pi m u}{l}}\, du=l\int_{-1}^0F_{-n}^l(.,s)e^{2im\Re \zeta}\, d\Re \zeta=\beta_{m,n}(s)N_s^m(\phi)$$ and $u=\ln|z|=-l\Re \zeta$, $\phi=\arg z=\pi -l\Im\zeta$.

With \ref{} we have
$$I_{m,n}\sim 
\frac{\Gamma(s+\bar{n})\Gamma(s-\bar{n}) \Gamma(s)^2 }{\pi 4^{1-s}\Gamma(2s)\Gamma(2s-1)} l^s(4\pi|n|)^{-s}\left(\frac{-4\pi|n|\Gamma(2s)}{\Gamma(s)}\right)\int_{-1}^0F_n(\zeta,s) e^{2im\Re \zeta}\, d\Re \zeta\, .$$

Now 
 $$ N_s^m(\pi -l \Im\zeta)=\frac{2\pi\sin\frac{\pi}{2}(s-\bar{m})\sin\frac{\pi}{2}(s+\bar{m}) }{\cos^2\pi s (s-\frac{1}{2})\Gamma(s-\frac{1}{2})\Gamma(\frac{1}{2}-s) }\gamma_m(1-s) N_{1-s}^m(l \Im\zeta) +\frac{\cos(\pi\bar{m})}{\cos \pi s} N_s^m(l \Im\zeta)\, .$$
Then
$$u_{m,n}(l)=\frac{I_{m,n}}{\frac{2\pi\sin\frac{\pi}{2}(s-\bar{m})\sin\frac{\pi}{2}(s+\bar{m}) }{\cos^2\pi s (s-\frac{1}{2})\Gamma(s-\frac{1}{2})\Gamma(\frac{1}{2}-s) }+\frac{\cos(\pi\bar{m})}{\cos \pi s} \frac{N_s^m(l \Im\zeta)}{\gamma_m(1-s)}}\, .$$
We have (\cite{borth},p. 164)
$$\gamma_m(1-s)\sim 2^{2s-1}\frac{\Gamma(s-\frac{1}{2})}{\Gamma(\frac{1}{2}-s)}\left|\frac{2\pi m}{l}\right|^{1-2s}$$
For the numerator we have
$$I_{m,n}\sim -4^s\pi^{s}e^{-2s}\frac{\Gamma(s)}{\Gamma(2s-1)}\left(\frac{|n|}{l}\right)^{s-1}|n|\int_{-1}^0F_n(\zeta,s) e^{2im\Re \zeta}\, d\Re\zeta \, e^{-\frac{2\pi^2|n|}{l}}\, .$$
For the numerator we have
$$\frac{\pi^{2s-1}\Gamma(\frac{1}{2}-s)}{2\cos \pi s\Gamma(s-\frac{1}{2})}|m|^{2s-1}e^{2\pi^2\frac{|m|}{l}} l^{1-s}\Im^s\zeta$$

\begin{proof}
We have $\tau_{S_l}(s)=\sum_{k=0}^\infty \tr \Lambda^k A_l$.

To estimate the preceding traces, we use Fourier expansion in \cite{fayreine}, p.12 and the following result is used in particular (\cite{borth}).

%
%
%
%
%

In the crowd we will obtain results of convergence of the scattering matrix as well as the Eisenstein series corresponding to the funnel $F_l$ for $\Re s>1/2$.
\end{proof}

%
%
%
%

We develop here an explicit case for Theorem \ref{clef} and look more in details when $C_l=\Gamma_l\backslash H$ is the hyperbolic cylinder. The limited surface is to copies of the parabolic cylinder $C_\infty=\Gamma_\infty\backslash H$.

We begin to notice that the hyperbolic Eisensein series associated to the simple closed geodesic in $C_l$ is $E_l=\sin^s \arg z$ and the Eisenstein series in $C_\infty$ is $E_\infty(z)=\Im^s(z)$. If we do the change of variables to study the right (resp. left)  side of th collar (resp. left) $w=\frac{1}{l}\log z$ (resp. $\zeta=-\frac{1}{l}\log(-z)$), $\arg z=l\Im w$ (resp. $\arg z=\pi -l\Im \zeta$), we find that $\frac{1}{l^s}E_l\to E_\infty$.

Now we are going to explicit the generalised eigenfunction $E^{f_l}(s,z,t')=(1-2s)\lim_{\rho'\to 0}\rho'^{-s}G_s^{C_l}(z,z')$. For this we use \cite{borth}, Proposition 5.3, p.66

$$E^{f_l}(s,z,t')=\frac{2s-1}{2\sqrt{\pi}\Gamma(s+\frac{1}{2})}\sum_{k\in\ZZ}a_k(s)v_k(s;-r)e^{\bar{k}(t-t')}\,$$
with

$v_k(s;-r)=\frac{\sinh r(\cosh r)^{\bar{k}}}{\Gamma(\frac{s-\bar{k}}{2})\Gamma(\frac{s+\bar{k}}{2})}F(\frac{\bar{k}+1+s}{2}, \frac{\bar{k}+1+s}{2};\frac{3}{2}; -\sinh^2 r)+\frac{(\cosh r)^{\bar{k}}}{2\Gamma(\frac{1+s-\bar{k}}{2})\Gamma(\frac{1+s+\bar{k}}{2})}F(\frac{\bar{k}+s}{2}, \frac{\bar{k}+1-s}{2};\frac{1}{2}; -\sinh^2 r) $
where 

$a_k(l)=\frac{\pi}{l}4^{1-s}\Gamma(s-\bar{k})\Gamma(s+\bar{k})$.

As we have Proposition \ref{r} we obtain the Fourier coefficients for the hyperbolic cylinder
$$F^l_{-n}(z,s)=\frac{2s-1}{2\sqrt{\pi}\Gamma(s+\frac{1}{2})}a_n(s)v_n(s;-r)e^{\bar{n}t}\, .$$
From the  formula(\ref{co}), we have
$$\frac{(2s-1)l}{2\sqrt{\pi}\Gamma(s+\frac{1}{2})}a_n(s)v_n(s;-r)=N_{1-s}^n(\theta)+\beta_{n,n}(s)N_s^n(\theta)\, ,$$
with, as usual, $\theta=\arg z$, $1/\cosh r=\sin\theta$.
Using, for example, Formula 15.3.8 in \cite{abra}, we obtain
$$\beta_{n,n}(s)=\frac{1}{2}(\gamma_n(s)+\gamma_n^+(s))$$
where $\gamma_n(s)$ is given in (\ref{fun}) and 
$$\gamma_n^+(s))=\gamma_k(s)=\frac{\Gamma(\frac{1}{2}-s)\Gamma(\frac{s+1+\bar{k}}{2})\Gamma(\frac{s+1-\bar{k}}{2})}{\Gamma(s-\frac{1}{2})\Gamma(\frac{2-s+\bar{k}}{2})\Gamma(\frac{2-s-\bar{k}}{2})}\, .$$
\begin{rem}
Notice that we find again \cite{borth}, p. 115.
\end{rem}
Now we have
$$\tau_{C_l}(s)=\Pi_n \frac{1}{\gamma_n(s)}\beta_{n,n}(s)\Pi_n \frac{1}{2}(1+\frac{\gamma_n(s)}{\gamma_n^+(s)})$$

We know by Schulze that the quotient $\frac{Z_l(s)}{z_l(s)}$  converges to the Zeta function of the limit surface $S_0$ if $\Re s>1/2$. By the preceding theorem we can see that it is no more true in any domain that intersects $\{s\in \CC, \Re s<1/2\}$
\begin{cor}
The quotient $\displaystyle\frac{Z_l(s)}{z_l(s)}$ cannot converge for $\Re s<1/2$ as $l\to 0$.
\end{cor}
We will need the function equation connecting $Z_l(s)$ to $Z_l(1-s)$ via the relative scattering determinant. This is the analogue of Selberg's functional equation from the compact case.
Before writing it in our case, we recall some notations. Since there is exactly one primitive closed geodesic (with  two possible orientations) on the hyperbolic cylinder $C_l=\Gamma_l\backslash H$, the Selberg zeta function on $C_l$ is the entire function
$$z_{l}(s)=\Pi_{k\geq 0}(1-e^{-(s+k)l})^2\, .$$
Following \cite{borth}, p. 180, we will set, for the funnel zeta function, for $s\in\CC$
\begin{equation}
Z_{F_l}(s)=e^{-sl/4}\Pi_{k\geq 0}(1-e^{-(s+2k+1)l})^2\, .
\end{equation}
Let's verify that for $s\in\CC$
\begin{lemma}
$$\frac{Z_{F_l}(1-s)}{Z_{F_l}(s)}=O(l^{2s-1})\, .$$
\end{lemma}
\begin{proof}
We rewrite the absolute convergent product of the Selberg zeta function of the cylinder 
$$z_{l}(s)=\Pi_{k\geq 0}(1-e^{-(s+k)l})^2=\Pi_{k\geq 0}(1-e^{-(s+2k+1)l})^2\Pi_{k\geq 0}(1-e^{-(s+2k)l})^2\, .$$
Now
$$\Pi_{k\geq 0}(1-e^{-(s+2k)l})^2=\Pi_{k\geq 0}(1-e^{-(s/2+k)2l})^2=z_{2l}(s/2)\, .$$
We have \cite{schulze} Lemma 39 (see also \cite{he}, \cite{wp})
$$\lim_{l\to 0}z_l(s)\Gamma(s)^2e^{\pi^2/3l}l^{2s-1}=2\pi\, .$$
Then $z_{2l}(s/2)=O(e^{-\frac{\pi^2}{6l}}l^{1-s})$ and the result with
$$Z_{F_l}(s)=e^{-s l/4}z_l(s)/z_{2l}(s/2)=O(e^{-\frac{\pi^2}{6l}}l^{-s})\, .$$
\end{proof}
The entire function $G_\infty$ is defined in terms of the Barnes G-function (\cite{borth},p. 44)
\begin{eqnarray*}
G(s+1)&=&(2\pi)^{s/2}e^{-s/2-(\gamma+1)s^2/2}\Pi_{k=1}^\infty (1+\frac{s}{k})^ke^{-s+s^2/(2k)}\, ,\\
G_\infty&=&(2\pi)^{-s}\Gamma(s)G(s)^2
\end{eqnarray*}

\begin{proof}
We have \cite{borth}, p. 195

$$\tau_{S_l}(s)=c\frac{Z_{l}(1-s)}{Z_{l}(s)}\left(\frac{G_\infty(s)}{G_\infty(1-s)}\right)^{-\chi(S_l)}\frac{Z_{F_l}(1-s)}{Z_{F_l}(s)}$$
with a constant $c=\tau_{S_l}(1/2)$  and the Euler characteristic $\chi(S_l)=2-2g-n_f$. 
We have $\tau_{S_l}(s)\tau_{S_l}(1-s)=1$  then $|c|=1$. Moreover $\chi(S_l)$ is invariant under degeneration because we lost a funnel and win a cusp.

Then for $\Re s>1/2$

$$Z_l(1-s)=\frac{1}{c}\left(\frac{G_\infty(s)}{G_\infty(1-s)}\right)^{\chi(S_l)}\frac{Z_{F_l}(1-s)}{Z_{F_l}(s)}\tau_{S_l}(s)Z_l(s)$$

Moreover

$$z_l(1-s)=l^{4s-2}z_l(s)$$

and so
 $$\frac{Z_{l}(1-s)}{z_{l}(1-s)}=O(\frac{1}{l^{2s-1}})\frac{Z_{l}(s)}{z_{l}(s)}\, .$$

\end{proof}

We are going to see that, like in the finite volume case
\begin{prop}
 For $0<\Re s <1/2$, $\displaystyle\frac{1}{l^s} E_l(z,s)=0(l^{1-2s})$.
\end{prop}
 With this result we may hope the convergence of $E_l^*(z,s)$ below $\Re s=1/2$.

\begin{proof}
We can write
$$G_s(z,w)-G_{1-s}(z,w)=\frac{1}{1-2s}\int_0^l E^{f_l}(1-s,z,t')E^{f_l}(s,w,t')\, dt'\, .$$
Using Fourier expansion in \cite{fayreine}, we obtain
$$\frac{1}{l}\int_0^l E^{f_l}(1-s,z,t')E^{f_l}(s,w,t')\, dt'\sim \frac{1}{l^s}$$ and the result, namely: $\displaystyle\frac{1}{l^s} E_l(z,s)=0(l^{1-2s})$.
\end{proof}
%
%
%

%
%
%
%
%
%
%
%
%
%

Now, unlike the finite volume case, for $0<\Re s<1/2$, $\displaystyle\frac{Z_l(s)}{z_l(s)}=O(\displaystyle\frac{1}{l^{1-2s}})$, we deduce

\begin{thm}
The family $(E^*_l(.,s))_l$ converges uniformly on compact subsets of $S_0$ and $\{ \Re s >0\}$ to $Z_0(s) E_\infty(.,s)$.
\end{thm}

\begin{proof}
We know by (\cite{schulze} and \cite{hal}) that the family $(E^*_l(z,s))_l$ is bounded on  $\Re s>1/2$. From the preceding theorem, it's also bounded on $\{ 0<\Re s <1/2\}$. Then by Phragm\'en-Lindel\"{o}f method (see for exemple \cite{he}), it's bounded on $\Re s>0$ and by Vitali's theorem we have the result.
\end{proof}

From this study we hope to obtain some applications on  the spectrum of the Laplace operator.

\end{document}